\pdfoutput=1
\documentclass[a4paper,UKenglish,cleveref, autoref, thm-restate]{lipics-v2021}

\title{Neighborhood persistency of the linear optimization relaxation of integer linear optimization}

\titlerunning{Neighborhood persistency of integer linear optimization on UTVPI systems} 

\author{Kei Kimura
}{Faculty of Information Science and Electrical Engineering, Kyushu University, Japan \and \url{https://researchmap.jp/kei-kimura?lang=en} }{kkimura@inf.kyushu-u.ac.jp}{https://orcid.org/0000-0002-0560-5127}{Supported by JST, ACT-X Grant Number JPMJAX200C, Japan, and JSPS KAKENHI Grant Numbers JP19K22841, JP21K17700.}
\author{Kotaro Nakayama}{Oplan Incorporated, Japan}{kotaronakayama22@gmail.com}{}{}
\authorrunning{K. Kimura and K. Nakayama} 
\ccsdesc[100]{Theory of computation~Design and analysis of algorithms~Mathematical optimization~Discrete optimization} 

\keywords{integer linear optimization, linear optimization, unit-two-variable-per-inequality system, persistency} 

\category{} 

\relatedversion{} 

\nolinenumbers 

\hideLIPIcs  

\EventEditors{John Q. Open and Joan R. Access}
\EventNoEds{2}
\EventLongTitle{42nd Conference on Very Important Topics (CVIT 2016)}
\EventShortTitle{CVIT 2016}
\EventAcronym{CVIT}
\EventYear{2016}
\EventDate{December 24--27, 2016}
\EventLocation{Little Whinging, United Kingdom}
\EventLogo{}
\SeriesVolume{42}
\ArticleNo{23}

\begin{document}

\maketitle

\begin{abstract}
For an integer linear optimization (ILO) problem, persistency of its linear optimization (LO) relaxation is a property that for every optimal solution of the relaxation that assigns integer values to some variables, there exists an optimal solution of the ILO problem in which these variables retain the same values. Although persistency has been used to develop heuristic, approximation, and fixed-parameter algorithms for special cases of ILO, its applicability remains unknown in the literature. In this paper we reveal a maximal subclass of ILO such that its LO relaxation has persistency. Specifically, we show that the LO relaxation of ILO on unit-two-variable-per-inequality (UTVPI) systems has persistency and is (in a certain sense) maximal among such ILO. Our persistency result generalizes the results of Nemhauser and Trotter, Hochbaum et al., and Fiorini et al. Even more, we propose a stronger property called \emph{neighborhood persistency} and show that the LO relaxation of ILO on UTVPI systems in general has this property. Using this stronger result, we obtain a fixed-parameter algorithm (where the parameter is the solution size) and another proof of two-approximability for ILO on UTVPI systems where objective functions and variables are non-negative.
\end{abstract}

\section{Introduction}
\label{sec:introduction}

In this paper, we mainly investigate the ILO problem on a unit-two-variable-per-inequality (UTVPI) system.
In this problem, we are given matrix $A \in \{ -1,0,1 \}^{m \times n}$ with at most two nonzero elements per row, 
integer vector $b \in \mathbb{Z}^m$, and rational vector $w \in \mathbb{Q}^n$, and our task is to 
compute the optimal value of the following ILO problem:
\begin{align}\label{eq:ILO-intro}
\begin{array}{ll}
\text{minimize} & w^\mathrm{T}x\\
\text{subject to} & Ax \ge b,\\
 & x \in \mathbb{Z}^n.
\end{array}
\end{align}
ILO on UTVPI systems has many applications in practice and theory.
Practical applications include map labeling~\cite{BKP11} and scheduling~\cite{UpC13}, and 
theoretical applications include various problems in graph theory and combinatorial optimization such as the vertex cover problem, the maximum independent set problem, a disjoint path problem~\cite{Sch91}, 
and the minimum clique cover problem~\cite{BOS13}.
ILO on UTVPI systems is strongly NP-hard, since it includes the vertex cover problem, which is NP-hard.
When UTVPI systems are \emph{monotone}, i.e., each constraint is of the form $x_p -x_ q \ge c$, 
they are sometimes called \emph{difference constraint systems} (DCSs), and ILO on DCSs is solvable in polynomial time by minimum cost flow algorithms (see, e.g.,~\cite{AMO93}).
ILO on DCSs includes the dual linear optimization problem of the shortest path problem, a fundamental problem in combinatorial optimization.
The feasibility problem of UTVPI systems has also been extensively studied. 
The feasibility problem is solvable in polynomial time, and many algorithms have been proposed (\cite{Sch91,JMS94,Sub04,LaM05,SuW17}).
The feasibility problem also appears in practice, e.g., abstract interpretation~\cite{Min06}.

It is quite common to solve an ILO problem by first solving its \emph{linear optimization (LO) relaxation} and rounding up or down the obtained LO solution.
Here the linear optimization relaxation of an ILO problem is a problem where the integrality condition of the variables (i.e., $x \in \mathbb{Z}^n$) is dropped (or changed to $x \in \mathbb{R}^n$).
An optimal solution of the LO relaxation is sometimes called an optimal \emph{fractional} solution of the ILO problem, whose 
optimal solution is called an optimal \emph{integer} solution.

The LO relaxation of certain ILO subclasses has \emph{persistency}.
Persistency of the LO relaxation of an ILO problem is a property that for every optimal fractional solution that assigns integer values to some variables, 
an optimal \emph{integer} solution exists in which these variables retain the same values.
If the LO relaxation of an ILO problem has persistency, 
one can solve the problem by obtaining an optimal fractional solution (in polynomial time) by solving linear optimization, 
substituting the integer values of the fractional solution to the corresponding variables, 
and solving the resulting problem with fewer variables.
This algorithmic framework gives not only a fast heuristic algorithm but also a theoretically fast one.
Indeed, persistency was first shown for the LO relaxation of an ILO formulation of the vertex cover problem~\cite{NeT75}, which is ILO on special UTVPI systems, and 
used to obtain a fixed-parameter algorithm for the vertex cover problem~\cite{LNR14}.
The persistency result in \cite{NeT75} is generalized to special cases of ILO on UTVPI systems~\cite{HMN93,FJW21}.

\subsection*{Our contribution}

In this paper, we show that the LO relaxation of ILO on UTVPI systems in general has persistency.
More strongly, we propose \emph{neighborhood persistency}, which is stronger than persistency, and show 
that the LO relaxation of ILO on UTVPI systems in general has neighborhood persistency.\footnote{Neighborhood persistency resembles the generalization of persistency~\cite{FJW21}.
However, that work assumed that a solution of the relaxation problem is extreme, although it is not in the neighborhood persistency.}
Neighborhood persistency is a property that for every optimal fractional solution $x^*$, 
there exists an optimal integer solution in the \emph{integer neighborhood} of $x^*$, 
where for vector $x \in \mathbb{R}^n$ its integer neighborhood $N(x) \subseteq \mathbb{Z}^n$ is defined as 
$N(x)=\{ z \in \mathbb{Z}^n \mid |z_j -x_j| < 1\ (j=1,\dots, n) \}$.
To obtain our result, 
we use the celebrated strong duality theorem on LO, 
which was not used in the proofs of the previous results on persistency mentioned above.

We also show that ILO on UTVPI systems is a maximal subclass of ILO with its LO relaxation having persistency in the sense that 
if we allow (i) an inequality with \emph{three} variables whose coefficients are all one or (ii) 
an inequality with two variables whose coefficients are in $\{1,2\}$, then not even persistency holds for the LO relaxation of such ILO.

From our neighborhood-persistency result, we can solve ILO on UTVPI systems by 
(i) solving the LO relaxation 
and then (ii) solving a \emph{binary} ILO problem (i.e., each variable takes value zero or one), since an optimal integer solution (if it exists) can be obtained by rounding up or down the fractional values of the 
optimal LO solution.
Using this two-step algorithm we show fixed-parameter tractability (in terms of the solution size) of the ILO on UTVPI systems with non-negative objective functions and non-negative variables.
We also obtain another proof of two-approximability of such ILO problems, which was first shown by Hochbaum et al.~\cite{HMN93}.
Note that a (half-integral) optimal solution of the LO relaxation of an ILO problem on a UTVPI system can be efficiently computed 
by first transforming it to an ILO problem on a DCS by a previously proposed method \cite{HMN93} and solving the transformed ILO problem by a minimum cost flow algorithm.

\subsection*{Previous and Related work}
The vertex cover problem is, given an undirected graph $G=(V,E)$, to compute the minimum size of vertex subset $C \subseteq V$ such that every edge in $E$ has at least one end vertex in $C$.
It is well-known that this problem is formulated as ILO on special UTVPI systems as follows.
The variable set is $\{x_i \mid v_i \in V \}$, where each variable is binary, 
the objective function is $\sum_{i=1}^{|V|} x_i$, and the linear system is $\{ x_i + x_j \ge 1 \mid \{i,j\} \in E \}$.
Nemhauser and Trotter~\cite{NeT75} showed that the LO relaxation of this ILO formulation has persistency.
Generalizing this result, Hochbaum et al.~\cite{HMN93} showed that persistency also holds for the LO relaxation of ILO on UTVPI systems if the variables are binary and the coefficients in the objective function are non-negative.
Note that persistency and neighborhood persistency are the same when the variables are binary.
Fiorini et al.~\cite{FJW21} gave another generalization that persistency holds for the LO relaxation of ILO on UTVPI systems if each inequality is of the form $x_i + x_j \le c$ for some integer $c$.
It should be noted that optimal solutions of the LO relaxation are assumed to be half-integral in these persistency results, 
while not in our (neighborhood) persistency result in this paper.

A previous work \cite{HMN93} showed that one can obtain a two-approximate solution of the (feasible) ILO problem~\eqref{eq:ILO} by rounding up or down a half-integral optimal solution of the LO relaxation if $w$ is non-negative and the variables take non-negative values. 

The work \cite{HMN93} also showed that any ILO problem on a two-variable-per-inequality system (i.e., input matrix $A$ is in $\mathbb{Q}^{m \times n}$ and each row of $A$ has at most two nonzero elements) can be reduced to a binary ILO problem on a UTVPI system of pseudo-polynomial size if upper and lower bounds exists on the value of each variable. 
From the reduction and persistency of the LO relaxation of binary ILO on UTVPI systems (for non-negative objective functions), we can obtain upper and lower bounds on the values of the variables in an optimal integer solution of an ILO problem on a two-variable-per-inequality system. 
However, this does not imply neighborhood persistency of the LO relaxation of ILO on UTVPI systems.

The persistency of the LO relaxation of the ILO formulation of the vertex cover problem is generalized to the so-called $k$-submodular relaxation \cite{IWY16} where $k$ is any positive integer.
In $k$-submodular relaxation, a problem with values in $\{1,\dots,k\}$ is relaxed to one with values in $\{0,1,\dots,k\}$.
Although exactly one relaxed value (i.e., value $0$) exists in $k$-submodular relaxation, 
there exists infinite relaxed values (i.e., all the values in $\mathbb{R} \setminus \mathbb{Z}$) in our relaxation of ILO to LO.

Recently, Hirai~\cite{Hir18} gave a general result on persistency in discrete convex analysis on graph structures.
We investigate the relation between this and our persistency result and show that a slightly weaker version of our main theorem can be shown using this general result.
Details are found in \cref{appendix:L-extendability}.

\subsection*{Outline}
The rest of our paper is organized as follows.
Section~\ref{sec:preliminaries} formally defines our problem and (neighborhood) persistency, and provides useful results.
Section~\ref{sec:main-result} shows our main result, namely, neighborhood persistency of the LO relaxation of ILO on UTVPI systems.
Section~\ref{sec:maximality} gives examples of ILO problems on non-UTVPI systems such that their LO relaxations lack persistency, 
which shows the maximality of ILO on UTVPI systems among ILO with its LO relaxation having (neighborhood) persistency.
Section~\ref{sec:approx-FPT} shows that ILO on UTVPI systems with non-negative objective functions and non-negative variables is fixed-parameter tractable and two-approximable by using neighborhood persistency.
Section~\ref{sec:conclusion} concludes our paper.

\section{Preliminaries}\label{sec:preliminaries}
Let $\mathbb{Z}$, $\mathbb{Q}$, and $\mathbb{R}$ denote the sets of integers, rationals, and reals, respectively.

We consider an integer linear optimization (ILO) problem of the following form throughout the paper:
\begin{align}\label{eq:ILO}
\begin{array}{ll}
\text{minimize} & w^\mathrm{T}x\\
\text{subject to} & Ax \ge b,\\
 & x \in \mathbb{Z}^n, 
\end{array}
\end{align}
where $A \in \{ -1,0,1 \}^{m \times n}$ is a matrix having at most two nonzero elements per row,
$b \in \mathbb{Z}^m$, $w \in \mathbb{Q}^n$, and $m,n$ are positive integers.
A vector $x \in \mathbb{Z}^n$ satisfying $Ax \ge b$ is called a \emph{feasible} solution of the ILO problem~\eqref{eq:ILO}.
If the ILO problem~\eqref{eq:ILO} has a feasible solution, then it is called \emph{feasible}.
A feasible solution of the ILO problem~\eqref{eq:ILO} is called an \emph{optimal integer solution} or \emph{optimal solution} (when it is clear from context) 
if it has the minimum objective value among the feasible solutions.

The following linear optimization (LO) problem is called the \emph{LO relaxation} of the ILO problem~\eqref{eq:ILO}:

\begin{align}\label{eq:LO}
\begin{array}{ll}
\text{minimize} & w^\mathrm{T}x\\
\text{subject to} & Ax \ge b,\\
 & x \in \mathbb{R}^n.
\end{array}
\end{align}
Feasibility in the LO problem~\eqref{eq:LO} is defined analogously to that in the ILO problem~\eqref{eq:ILO}.
A feasible solution of the LO problem~\eqref{eq:LO} is called an \emph{optimal fractional solution} or \emph{optimal solution} (when it is clear from context) 
if it has the minimum objective value among the feasible solutions.

Now, we provide the key notions of this paper.

\begin{definition}[Persistency]
LO relaxation \eqref{eq:LO} is \emph{persistent} if 
for every optimal fractional solution that assigns integer values to some variables, 
there exists an optimal integer solution of \eqref{eq:ILO} in which these variables retain the same values.
Namely, for every optimal fractional solution $x^*$, there exists an optimal integer solution $z^*$ such that 
$x^*_j \in \mathbb{Z}$ implies $x^*_j = z^*_j$ for each $j =1, \dots, n$.
\end{definition}

\begin{definition}[Integer Neighborhood]
For vector $x \in \mathbb{R}^n$, 
its \emph{integer neighborhood} $N(x) \subseteq \mathbb{Z}^n$ is defined as 
$N(x)=\{ z \in \mathbb{Z}^n \mid |z_j -x_j| < 1\ (j=1,\dots, n) \}$.
\end{definition}

We focus on the following property, which is stronger than persistency.

\begin{definition}[Neighborhood Persistency]
LO relaxation \eqref{eq:LO} is \emph{neighborhood persistent} if 
for every optimal fractional solution $x^*$, 
there exists an optimal integer solution of \eqref{eq:ILO} in integer neighborhood $N(x^*)$ of $x^*$.
\end{definition}

Note that for vector $x \in \mathbb{R}^n$ and $j \in \{1,\dots, n\}$ if $x_j \in \mathbb{Z}$, then $z_j = x_j$ for any $z \in N(x)$.
Thus, neighborhood persistency implies (ordinary) persistency.

We will show our main result using the following form of the strong duality of LO.

\begin{theorem}[Theorem 5.4 in~\cite{Sch03}]
\label{thm:strong-duality}
Let $A$ be an integer matrix, let $b$ be an integer vector, and let $w$ be a rational vector.
If at least one of $\min \{ w^\mathrm{T}x \mid Ax \ge b, x \in \mathbb{R}^n \}$ or 
$\max \{ b^\mathrm{T}y \mid A^\mathrm{T}y = w, y\ge 0, y \in \mathbb{R}^m \}$ is bounded, 
then $\min \{ w^\mathrm{T}x \mid Ax \ge b, x \in \mathbb{R}^n \} = \max \{ b^\mathrm{T}y \mid A^\mathrm{T}y = w, y\ge 0, y \in \mathbb{R}^m \}$.
\end{theorem}

For $a \in \mathbb{R}$, define $\lceil a \rceil = \min\{ n \in \mathbb{Z} \mid n \ge a \}$ and 
$\lfloor a \rfloor = \max\{ n \in \mathbb{Z} \mid n \le a \}$.
We use the following easy-to-prove fact in some proofs of our results.

\begin{lemma}\label{lem:rounddown-property}
For $a \in \mathbb{R}$ and $b \in \mathbb{Z}$, 
if $a \ge b$, then $\lfloor a \rfloor \ge b$.
\end{lemma}

\begin{proof}
Let $a = \alpha + \beta$, where $\alpha \in \mathbb{Z}$ and $0 \le \beta < 1$.
Assume that $a \ge b$.
Then $\alpha + \beta \ge b$, implying that $\alpha \ge b - \beta > b - 1$.
Hence, $\alpha$ is an integer greater than $b-1$ and thus $\alpha \ge b$.
Since $\lfloor a \rfloor = \alpha$, we have $\lfloor a \rfloor \ge b$, as desired.
\end{proof}

\section{Main results}\label{sec:main-result}

In this section, we show the following theorem, which is the main result of this paper.

\begin{theorem}
\label{thm:main}
If the integer linear optimization problem~\eqref{eq:ILO} is feasible, 
then its linear optimization relaxation~\eqref{eq:LO} is neighborhood persistent.
\end{theorem}

\begin{proof}
Assume that the ILO problem~\eqref{eq:ILO} is feasible.
If the LO relaxation~\eqref{eq:LO} does not have an optimal fractional solution (i.e., it is unbounded), then 
the condition of the neighborhood persistency of the LO relaxation vacuously holds.
Therefore, we assume that the LO relaxation~\eqref{eq:LO} has an optimal fractional solution in what follows.
Then the ILO problem~\eqref{eq:ILO} is bounded and has an optimal integer solution.
Fix an optimal integer solution $z^*$ and an optimal fractional solution $x^*$ in what follows.
We show that there exists an optimal integer solution of the ILO problem~\eqref{eq:ILO} in integer neighborhood $N(x^*)$ of $x^*$, 
which shows the theorem.

Define $z \in \mathbb{Z}^n$ as 
\begin{align}
z_j = \begin{cases}
x^*_j & (x^*_j \in \mathbb{Z}),\\
\lceil x^*_j \rceil & (x^*_j \not \in \mathbb{Z}\text{ and } x^*_j < z^*_j),\\
\lfloor x^*_j \rfloor & (x^*_j \not \in \mathbb{Z}\text{ and } x^*_j > z^*_j),
\end{cases}
\end{align}
for $j= 1, \dots, n$.
Note that $z \in N(x^*)$. 
We show that $z$ is an optimal integer solution of the ILO problem~\eqref{eq:ILO}.
For this, we show that (i) $z$ is a feasible solution of the ILO problem~\eqref{eq:ILO} (in Claim~\ref{cl:z-feasibility} below), and 
(ii) $w^\mathrm{T}z \le w^\mathrm{T}z^*$ (in Claim~\ref{cl:z-optimality} below).
These imply that $z$ is an optimal integer solution of the ILO problem~\eqref{eq:ILO}, 
since so is $z^*$.

\begin{claim}\label{cl:z-feasibility}
$z$ is a feasible solution of the ILO problem~\eqref{eq:ILO}.
\end{claim}
\begin{proof}
We show that $Az \ge b$ holds.
For each $i = 1, \dots, m$, let $A_i$ be the $i$th row of $A$.
We show that $A_iz \ge b_i$ for each $i = 1, \dots, m$.
Since $Ax \ge b$ is a UTVPI system, 
each $A_ix \ge b_i$ is of the form 
$\sigma x_p \ge b_i$ or $\sigma x_p + \tau x_q \ge b_i$ for some 
$\sigma,\tau \in \{-,+\}$ and $p,q \in \{1, \dots, n\}$ with $p \neq q$.
We divide the proof into single-variable and two-variable cases, and
use the fact that $A_ix^* \ge b_i$ and $A_iz^* \ge b_i$ (since $x^*$ and $z^*$ are feasible solutions) in what follows.

\paragraph*{Case 1: $A_ix = \sigma x_p$.}

\subparagraph*{Case 1.1: $\sigma z_p \ge \sigma x^*_p$.}

Since $\sigma x^*_p \ge b_i$, 
we have $\sigma z_p \ge \sigma x^*_p \ge b_i$.
Hence, we have $A_iz \ge b_i$.

\subparagraph*{Case 1.2: $\sigma z_p < \sigma x^*_p$.}

Since $z_p$ is an integer obtained by rounding $x^*_p$, we have $\sigma z_p = \lfloor \sigma x^*_p \rfloor$.
Moreover, the feasibility of $x^*$ implies that $\sigma x^*_p \ge b_i$, and thus $\lfloor \sigma x^*_p \rfloor \ge b_i$ by the integrality of $b_i$ and \cref{lem:rounddown-property}.
Hence, we have $\sigma z_p = \lfloor \sigma x^*_p \rfloor \ge b_i$, obtaining $A_iz \ge b_i$.
This completes the proof of Case 1.

\paragraph*{Case 2: $A_ix = \sigma x_p + \tau x_q$.}
We divide into the cases based on
the small and large comparison of the values $\sigma z_p,\tau z_q$ and $\sigma x^*_p,\tau x^*_q$, respectively.

\subparagraph*{Case 2.1: $\sigma z_p \ge \sigma x^*_p$ and $\tau z_q \ge \tau x^*_q$.}
We have $\sigma z_p + \tau z_q \ge \sigma x^*_p + \tau x^*_q \ge b_i$, 
obtaining $A_iz \ge b_i$.

\subparagraph*{Case 2.2: $\sigma z_p < \sigma x^*_p$ and $\tau z_q < \tau x^*_q$.}
In this case, we have $\sigma z_p = \lfloor \sigma x^*_p \rfloor$, as in Case 1.2.
Moreover, since $z_p$ is rounded toward $z^*_p$, we have $\sigma z_p \ge \sigma z^*_p$.
Similarly, we have $\tau z_q \ge \tau z^*_q$.
It follows that $\sigma z_p + \tau z_q \ge \sigma z^*_p + \tau z^*_q \ge b_i$, 
obtaining $A_iz \ge b_i$.

\subparagraph*{Case 2.3: $\sigma z_p \ge \sigma x^*_p$ and $\tau z_q < \tau x^*_q$.}
Let $\alpha = (\sigma z_p)- (\sigma x^*_p)$ and $\beta = (\tau x^*_q) - (\tau z_q)$.
By definition, we have $0 \le \alpha < 1$ and $0 < \beta < 1$.
If $\alpha \ge \beta$, then we have 
\begin{align}
\sigma z_p + \tau z_q = \sigma x^*_p + \tau x^*_q +(\alpha-\beta) \ge \sigma x^*_p + \tau x^*_q \ge b_i.
\end{align}
If $\alpha < \beta$, then we have $0 < \beta-\alpha < 1$, and since $\sigma x^*_p + \tau x^*_q = \sigma z_p + \tau z_q +(\beta - \alpha)$ and $\sigma z_p + \tau z_q \in \mathbb{Z}$, we have $\lfloor \sigma x^*_p + \tau x^*_q \rfloor = \sigma z_p + \tau z_q$.
Since $\lfloor \sigma x^*_p + \tau x^*_q \rfloor \ge b_i$ by \cref{lem:rounddown-property}, we have $\sigma z_p + \tau z_q = \lfloor \sigma x^*_p + \tau x^*_q \rfloor \ge b_i$, 
obtaining $A_iz \ge b_i$.

\subparagraph*{Case 2.4: $\sigma z_p < \sigma x^*_p$ and $\tau z_q \ge \tau x^*_q$.}
We can show $A_iz \ge b_i$ in a similar way as in Case 2.3.
This completes the proof of Case 2.

Since we have shown that $A_iz \ge b_i$ $(i = 1, \dots, m)$ for all the cases,
$z$ is a feasible solution of the ILO problem~\eqref{eq:ILO}.
This completes the proof.
\end{proof}

\begin{claim}\label{cl:z-optimality}
$w^\mathrm{T}z \le w^\mathrm{T}z^*$.
\end{claim}
\begin{proof}
To show $w^\mathrm{T}z \le w^\mathrm{T}z^*$, 
we use the duality theorem of linear optimization (see Theorem~\ref{thm:strong-duality}).
The following is the dual LO problem of \eqref{eq:LO}:

\begin{align}\label{eq:LO-dual}
\begin{array}{ll}
\text{maximize} & b^\mathrm{T}y\\
\text{subject to} & 
A^\mathrm{T} y  = w,\\
 & y \ge 0,\\
 & y \in \mathbb{R}^m.
\end{array}
\end{align}

Since we are assuming that the LO problem~\eqref{eq:LO} has an optimal solution, 
the dual LO problem~\eqref{eq:LO-dual} also has an optimal solution and the optimal values of LO problems~\eqref{eq:LO} and \eqref{eq:LO-dual} are the same by Theorem~\ref{thm:strong-duality}.
Fix an optimal solution $y^*$ of the LO problem~\eqref{eq:LO-dual} in what follows.
We have $A^\mathrm{T}y^* = w$ from the equality in the LO problem~\eqref{eq:LO-dual}.
Thus, $w^\mathrm{T}z$ can be rewritten as 
\begin{align}
w^\mathrm{T}z =  (y^*)^\mathrm{T}Az= \sum_{i=1}^{m}y^*_iA_iz ,
\end{align}
where we recall that $A_i$ is the $i$th row of $A$ for each $i = 1, \dots, m$.
Similarly, we have 
\begin{align}
w^\mathrm{T}z^* = \sum_{i=1}^{m}y^*_iA_iz^*.
\end{align}
Therefore, to show $w^\mathrm{T}z \le w^\mathrm{T}z^*$, it suffices to show that 
\begin{align}
\label{eq:sufficient-condition-for-wz<=wz*}
y^*_iA_iz \le y^*_iA_iz^*
\end{align}
for each $i=1,\dots, m$.
We show \eqref{eq:sufficient-condition-for-wz<=wz*} in what follows.

Recall that $x^*$ is an optimal solution of the LO problem~\eqref{eq:LO}.
The following condition called the \emph{complementary slackness} condition holds (see, e.g., Section 5.5 in~\cite{Sch03}):
\begin{align}
y^*_i(A_ix^*-b_i) = 0 
\end{align}
for each $i=1,\dots, m$.
For our purpose, we use the following equivalent form of the complementary slackness condition: 
\begin{align}\label{eq:complementary-slackness}
\text{If $y^*_i > 0$, then $A_ix^*-b_i = 0$}
\end{align}
for each $i=1,\dots, m$.
Now we are ready to show \eqref{eq:sufficient-condition-for-wz<=wz*} for each $i=1,\dots, m$.

Fix $i \in \{1,\dots, m\}$.
If $y^*_i = 0$, then $y^*_iA_iz \le y^*_iA_iz^*$ since $y^*_iA_iz = y^*_iA_iz^* = 0$.
Therefore, we assume that $y^*_i > 0$ in what follows.
Then it suffices to show that $A_iz \le A_iz^*$ for showing \eqref{eq:sufficient-condition-for-wz<=wz*}.
Note that we have $A_ix^*-b_i = 0$ from \eqref{eq:complementary-slackness}.

Since $Ax \ge b$ is a UTVPI system, 
$A_ix \ge b_i$ is of the form 
$\pm x_p \pm x_q \ge b_i$ or 
$\pm x_p \ge b_i$ for some $p,q \in \{1, \dots, n\}$ with $p \neq q$.
We divide the proof into cases by the (non-)integrality of $x^*_p$ and $x^*_q$.
The single-variable case is dealt with in Case 2 below.
Note that $A_iz^* \ge b_i$, since $z^*$ is a feasible solution.

\paragraph*{Case 1: $x^*_p \in \mathbb{Z}$ and $x^*_q \in \mathbb{Z}$.}
Since $z_p=x^*_p$ and $z_q=x^*_q$ by the definition of $z$, 
we have $A_iz = A_ix^* = b_i$.
Since we have $A_iz^* \ge b_i$, we have $A_iz \le A_iz^*$.

\paragraph*{Case 2: $x^*_p\in \mathbb{Z}$ and $x^*_q \not\in \mathbb{Z}$ (This case includes the case of $x^*_p \not\in \mathbb{Z}$ and $x^*_q \in \mathbb{Z}$ by the symmetry of the constraints).}

\subparagraph*{Case 2.1: $A_ix = \pm x_p \pm  x_q$.}
This case does not occur, since 
$x^*_p\in \mathbb{Z}$ and $x^*_q \not\in \mathbb{Z}$ imply that 
$A_ix^* \not\in \mathbb{Z}$, and from $b_i \in \mathbb{Z}$ 
we cannot have $A_ix^* = b_i$.

\subparagraph*{Case 2.2: $A_ix = \pm x_p$.}
Since $z_p=x^*_p$ by the definition of $z$, 
we have $\pm z_p = \pm x^*_p = b_i$.
Since we have $A_iz^* \ge b_i$, we have $A_iz \le A_iz^*$.

\subparagraph*{Case 2.3: $A_ix = \pm x_q$.}
This case does not occur, since $x^*_q \not\in \mathbb{Z}$ implies that 
$A_ix^* \not\in \mathbb{Z}$, and from $b_i \in \mathbb{Z}$ 
we cannot have $A_ix^* = b_i$.

\paragraph*{Case 3: $x^*_p \not\in \mathbb{Z}$ and $x^*_q \not\in \mathbb{Z}$.}
We divide into cases where $A_ix \ge b_i$ is 
$x_p + x_q \ge b_i$, 
$x_p - x_q \ge b_i$, 
$-x_p + x_q \ge b_i$, or
$-x_p - x_q \ge b_i$.
Further, we divide into cases by
the small and large comparison of the values $x^*_p,x^*_q$ and $z^*_p,z^*_q$.
Since $x^*_p,x^*_q \not\in \mathbb{Z}$ and $z^*_p,z^*_q \in \mathbb{Z}$, 
we have four cases: (i) $x^*_p < z^*_p$ and $x^*_q < z^*_q$, 
(ii) $x^*_p < z^*_p$ and $x^*_q > z^*_q$, 
(iii)  $x^*_p > z^*_p$ and $x^*_q < z^*_q$, or 
(iv)  $x^*_p > z^*_p$ and $x^*_q > z^*_q$.
Since these cases can be proven in similar ways, 
we only show the case of (i) $x^*_p < z^*_p$ and $x^*_q < z^*_q$, and 
omit the proof of the remaining cases.

In the following, 
we assume that $x^*_p < z^*_p$ and $x^*_q < z^*_q$.
Then we have $z_p \le z^*_p$ and $z_q \le z^*_q$ by the definition of $z$.
Let $x^*_p = \alpha_p + \beta_p$ and $x^*_q = \alpha_q + \beta_q$, where $\alpha_p,\alpha_q \in \mathbb{Z}$ and $0 < \beta_p,\beta_q < 1$.
Note that we have $z_p = \lceil x^*_p \rceil = x^*_p + 1 - \beta_p$ and $z_q = \lceil x^*_q \rceil = x^*_q + 1 - \beta_q$ by definition.

\subparagraph*{Case 3.1: $A_ix = x_p + x_q$.}
From $z_p \le z^*_p$ and $z_q \le z^*_q$, we have 
$z_p + z_q \le z^*_p + z^*_q$.
Hence, we have $A_iz \le A_iz^*$.

\subparagraph*{Case 3.2: $A_ix = x_p - x_q$.}
From $A_ix^* = b_i$, we have 
\begin{align}
x^*_p - x^*_q = \alpha_p - \alpha_q + \beta_p - \beta_q = b_i.
\end{align}
Since $b_i \in \mathbb{Z}$ and $-1 < \beta_p - \beta_q < 1$, 
we have $\beta_p - \beta_q = 0$.
Hence, we have 
\begin{align}
z_p - z_q = (x^*_p + 1 - \beta_p) - (x^*_q +1 - \beta_q) \\
= x^*_p - x^*_q - (\beta_p - \beta_q) = x^*_p - x^*_q = b_i.
\end{align}
Since $z^*_p - z^*_q \ge b_i$, 
we have $z_p - z_q \le z^*_p - z^*_q$, i.e., $A_iz \le A_iz^*$.

\subparagraph*{Case 3.3: $A_ix = -x_p + x_q$.}
We can show $A_iz \le A_iz^*$ in a similar way as in Case 3.2.

\subparagraph*{Case 3.4: $A_ix = -x_p - x_q$.}
Let $z^*_p = x^*_p + \gamma_p$ and $z^*_q = x^*_q + \gamma_q$, where $\gamma_p,\gamma_q > 0$.
Since $-x^*_p - x^*_q = b_i$ from $A_ix^* = b_i$, we have 
\begin{align}
-z^*_p - z^*_q = -(x^*_p + \gamma_p) - (x^*_q + \gamma_q) \\
= -x^*_p - x^*_q - (\gamma_p + \gamma_q) < b_i.
\end{align}
This contradicts that $z^*$ is feasible.
Hence, this case does not occur.
This completes the proof.
\end{proof}

From Claims~\ref{cl:z-feasibility} and \ref{cl:z-optimality}, 
we conclude that $z$ is an optimal integer solution of the ILO problem~\eqref{eq:ILO}.
This completes the proof of \cref{thm:main}.
\end{proof}

As we mentioned in Introduction, 
we can show a slightly weaker version of \cref{thm:main} using a general result on persistency in discrete convex analysis on graph structures~\cite{Hir18}.
A discussion on this is found in \cref{appendix:L-extendability}.

\begin{remark}
From (neighborhood) persistency, given an optimal solution of the LO relaxation~\eqref{eq:LO}, 
we can reduce the number of variables in the ILO problem\eqref{eq:ILO} problem by the number of variables that have integer values in the solution of the LO relaxation.
Hence, it is desirable to obtain an optimal solution of \eqref{eq:LO} in which the number of variables having integer values is maximum.
Here, we show that one can find in polynomial time an optimal solution of \eqref{eq:LO} in which the set of variables having integer values is \emph{maximal} as follows.
Let $x^*$ be an arbitrary optimal solution of \eqref{eq:LO}.
For each $j \in \{1,\dots, n\}$ with $x^*_j \in \mathbb{Z}$ fix the value of $x_j$ to $x^*_j$ to reduce the size of the LO problem.
Choose a $j \in \{1,\dots, n\}$ with $x^*_j \not \in \mathbb{Z}$ and 
solve two LO problems with the value of $x_j$ fixed to $\lfloor x^*_j \rfloor$ or $\lceil x^*_j \rceil$.
If the optimal value of one of the LO problems is the same as that of the original LO problem, then fix the value of $x_j$ accordingly.
This fixation is valid since if we have an optimal solution of the reduced LO problem where $x_j$ has an integer value, then we also have 
an optimal solution of the reduced LO problem where $x_j$ is either $\lfloor x^*_j \rfloor$ or $\lceil x^*_j \rceil$ by convexity of the set of optimal solutions of the reduced LO problem.
Repeating this process until no LO problem with a fixed variable has the same optimal value as the original one, 
we obtain an optimal solution of \eqref{eq:LO} in which the set of variables having integer values is maximal.
\end{remark}

\section{Maximality of UTVPI systems}\label{sec:maximality}
In this section, we show that ILO on UTVPI systems is a maximal subclass of ILO with its LO relaxation having (neighborhood) persistency in the following sense: 
If we allow (i) an inequality with \emph{three} variables whose coefficients are all one or 
(ii) an inequality with two variables whose coefficients are in $\{1,2\}$, 
then even persistency does not hold for the LO relaxation of the ILO problem~\eqref{eq:ILO}.

\begin{example}\label{example1}
Consider the following ILO problem:

\begin{align}\label{eq:example1}
\begin{array}{ll}
\text{minimize} & 3x_1+x_2\\
\text{subject to} & x_1 + x_2 + x_3 \ge 2,\\
& x_1 - x_3 \ge 0,\\
 & -x_2 \ge -1,\\
 & x \in \mathbb{Z}^3.
\end{array}
\end{align}

By using an (I)LO solver, 
one can check that $(x^*_1,x^*_2,x^*_3) = (0.5,1,0.5)$ is an optimal fractional solution of the LO relaxation of \eqref{eq:example1}, and 
$(z^*_1,z^*_2,z^*_3) = (1,0,1)$ is an optimal integer solution of \eqref{eq:example1}, whose objective value is 3.
On the other hand, if we fix $x_2 = x^*_2 = 1$ in \eqref{eq:example1}, 
then we obtain the following ILO problem:

\begin{align}\label{eq:example1-1}
\begin{array}{ll}
\text{minimize} & 3x_1+1\\
\text{subject to} & x_1 + x_3 \ge 1,\\
& x_1 - x_3 \ge 0,\\
 & x_1,x_3 \in \mathbb{Z}.
\end{array}
\end{align}

By using an ILO solver, 
one can check that the ILO problem~\eqref{eq:example1-1} has an optimal integer solution $(z^*_1,z^*_3) = (1,0)$ whose objective value is 4.
If the LO relaxation of the ILO problem~\eqref{eq:example1} has persistency, then the optimal values of the ILO problems~\eqref{eq:example1} and  \eqref{eq:example1-1} must be 
equal. 
However, they are different and 
we conclude that the LO relaxation of \eqref{eq:example1} does not have persistency.
\end{example}

\begin{example}\label{example2}
Consider the following ILO problem:

\begin{align}\label{eq:example2}
\begin{array}{ll}
\text{minimize} & 3x_1+x_2\\
\text{subject to} & 2x_1 + x_2 \ge 2,\\
& -x_2 \ge -1,\\
& x \in \mathbb{Z}^2.
\end{array}
\end{align}

By using an (I)LO solver, 
one can check that $(x^*_1,x^*_2) = (0.5,1)$ is an optimal fractional solution of the LO relaxation of \eqref{eq:example2}, and 
$(z^*_1,z^*_2) = (1,0)$ is an optimal integer solution of \eqref{eq:example2}, whose objective value is 3.
On the other hand, if we fix $x_2 = x^*_2 = 1$ in \eqref{eq:example2}, 
then we obtain the following ILO problem:

\begin{align}\label{eq:example2-1}
\begin{array}{ll}
\text{minimize} & 3x_1+1\\
\text{subject to} & 2x_1 \ge 1,\\
 & x_1 \in \mathbb{Z}.
\end{array}
\end{align}

The ILO problem~\eqref{eq:example2-1} has the unique optimal integer solution $z^*_1 = 1$ whose objective value is 4.
Since the optimal values of ILO problems~\eqref{eq:example2} and \eqref{eq:example2-1} differ, 
we conclude that the LO relaxation of \eqref{eq:example2} does not have persistency.
\end{example}

From Examples~\ref{example1} and \ref{example2}, together with Theorem~\ref{thm:main}, 
we see that ILO on UTVPI systems is a maximal subclass of ILO with its LO relaxation having (neighborhood) persistency.

From Examples~\ref{example1} and \ref{example2}, 
we also see that persistency does not hold even for the LO relaxation of \emph{binary} ILO when we allow one of the inequalities specified above.
Therefore, binary ILO on UTVPI systems is a maximal subclass of binary ILO with its LO relaxation having (neighborhood) persistency.

\section{Fixed-parameter tractability and two-approximability for special cases}
\label{sec:approx-FPT}
In this section, we consider ILO on UTVPI systems with non-negative objective functions and non-negative variables and
address
the following ILO problem: 
\begin{align}\label{eq:ILO2}
\begin{array}{ll}
\text{minimize} & w^\mathrm{T}x\\
\text{subject to} & Ax \ge b,\\
 & x \ge 0,\\
 & x \in \mathbb{Z}^n, 
\end{array}
\end{align}
where $A \in \{ -1,0,1 \}^{m \times n}$ is a matrix having at most two nonzero elements per row, 
$b \in \mathbb{Z}^m$, $w \in \mathbb{Q}^n_{+}$ (where $\mathbb{Q}_+$ denotes the set of non-negative rationals), and $m,n$ are positive integers.
We show that the ILO problem \eqref{eq:ILO2} is both fixed-parameter tractable and two-approximable in what follows.

From our main result (\cref{thm:main}) we can reduce solving the ILO problem \eqref{eq:ILO2} to solving an ILO problem with \emph{binary} variables, i.e., each variable takes value zero or one.
Indeed, let $x^*$ be an optimal fractional solution of the LO relaxation of \eqref{eq:ILO2} and let $\lfloor x^* \rfloor$ be a vector obtained from $x^*$ by taking componentwise $\lfloor \cdot \rfloor$.
Then from \cref{thm:main} the ILO problem \eqref{eq:ILO2} is equivalent to the following problem:
\begin{align}\label{eq:ILO3}
\begin{array}{ll}
\text{minimize} & (w')^\mathrm{T}x'+w^\mathrm{T}\lfloor x^* \rfloor\\
\text{subject to} & A'x'+A\lfloor x^* \rfloor \ge b,\\
 & x' \in \{0,1\}^{n-|I(x^*)|},
\end{array}
\end{align}
where $I(x^*) = \{ j\in \{1,\dots,n\} \mid x^*_j \in \mathbb{Z} \}$, 
and $w'$ (resp, $A'$) is a restriction of $w$ (resp., columns of $A$) to $\{1,\dots,n\} \setminus I(x^*)$.
In turn, the ILO problem \eqref{eq:ILO3} is equivalent to solving 
\begin{align}\label{eq:ILO4}
\begin{array}{ll}
\text{minimize} & (w')^\mathrm{T}x'\\
\text{subject to} & A'x' \ge b-A\lfloor x^* \rfloor,\\
 & x' \in \{0,1\}^{n-|I(x^*)|}.
\end{array}
\end{align}

The ILO problem \eqref{eq:ILO4} with binary variables is 
fixed-parameter tractable~\cite{MNR13}\footnote{ILO on UTVPI systems with non-negative objective functions and binary variables is equivalent to the weighted min one 2-SAT problem in \cite{MNR13}.} (i.e., there exists an algorithm that solves a problem with parameter $k$ in time $f(k)s^{{\rm O}(1)}$ where $s$ is the input size of the problem) and 
two-approximable~\cite{HMN93} (i.e., there exists a polynomial time algorithm that outputs a feasible solution (if it exists) where the objective value is at most twice the optimal value).
We show that we can obtain the same results for the (non-binary) ILO problem \eqref{eq:ILO2}.

Let $k$ be a positive integer.
From \cref{thm:main}, 
the ILO problem \eqref{eq:ILO2} has a solution whose objective value is at most $k$ if and only if  
the ILO problem \eqref{eq:ILO4} has a solution whose objective value is at most $k-w^T\lfloor x^* \rfloor (\le k)$.
Moreover, the optimal value of the LO relaxation of \eqref{eq:ILO2} is the sum of the optimal value of the LO relaxation of \eqref{eq:ILO4} and $w^T\lfloor x^* \rfloor$.
Consequently, the following results on fixed-parameter tractability hold from previous results \cite{MNR13}.
Notation ${\rm O}^*()$ hides functions that are polynomial in the input size in what follows.

\begin{theorem}
Given an ILO problem \eqref{eq:ILO2} such that $w_j \ge 1$ for all $j \in \{1,\dots,n\}$ and positive integer $k$,
it can be verified in time ${\rm O}^*(1.3788^k)$ if \eqref{eq:ILO2} has a feasible solution whose objective value is at most $k$.
An optimal integer solution (if it exists) can be obtained in time ${\rm O}^*(1.2377^n)$.
\end{theorem}

When $w$ is an all-one vector (i.e., $w_j = 1$ for all $j \in \{1,\dots,n\}$), 
we obtain the following:

\begin{theorem}
Let $I$ be an ILO problem \eqref{eq:ILO2} such that $w_j = 1$ for all $j \in \{1,\dots,n\}$ and $k$ be a positive integer.
Then 
\begin{itemize}
\item it can be checked in time ${\rm O}^*(1.2738^k)$ whether $I$ has a feasible solution whose objective value is at most $k$.
An optimal integer solution (if it exists) can be obtained in time ${\rm O}^*(1.2114^n)$ and polynomial space, or ${\rm O}^*(1.2108^n)$ and exponential space; 
\item it can be checked in time ${\rm O}^*(2.3146^{k-{\rm OPT}_{\rm LO}})$ whether $I$ has a feasible solution whose objective value is at most $k$ where ${\rm OPT}_{\rm LO}$ is the optimal value of the LO relaxation of $I$; 
\item there exists a randomized polynomial time algorithm that produces ILO problem $I'$ on a UTVPI system with an all-one objective function vector and binary variables, and $k'$ such that $I'$ has a number of variables and inequalities polynomial in $k-{\rm OPT}_{\rm LO}$ and 
if $I$ has a feasible solution whose objective value is at most $k$, then $I'$ has a feasible solution whose objective value is at most $k'$, and 
if $I$ has no feasible solution with the objective value at most $k$, then with the probability at least half, $I'$ has no feasible solution whose objective value is at most $k'$.
\end{itemize}
\end{theorem}

For approximability, we obtain the following:

\begin{theorem}
ILO on UTVPI systems with non-negative objective functions and non-negative variables is 2-approximable.
\end{theorem}
\begin{proof}
Let ${\rm OPT}_{1}$ (resp., ${\rm OPT}_{2}$) be the optimal value of the ILO problem \eqref{eq:ILO2} (resp., \eqref{eq:ILO4}).
Let $x'$ be a two-approximate solution of the ILO problem \eqref{eq:ILO4} obtained by a previously proposed algorithm~\cite{HMN93}.
Define $x \in \mathbb{Z}^n$ as $x_j=\lfloor x^*_j \rfloor + x'_j$ if $j \not\in I(x^*)$ and $x_j=\lfloor x^*_j \rfloor$ otherwise (i.e. if $j \in I(x^*)$).
Then $x$ is a two-approximate solution of the ILO problem \eqref{eq:ILO2}, since 
\begin{align}
w^Tx &= w^T \lfloor x^* \rfloor + (w')^Tx' \\
&\le w^T \lfloor x^* \rfloor + 2{\rm OPT}_2\\
&\le 2w^T \lfloor x^* \rfloor + 2{\rm OPT}_2\\
&= 2{\rm OPT}_1, 
\end{align}
where $w^T \lfloor x^* \rfloor + {\rm OPT}_2 = {\rm OPT}_1$ from \cref{thm:main}.
This completes the proof.
\end{proof}

Two-approximability of ILO on UTVPI systems with non-binary variables is already known~\cite{HMN93}.
Therefore, we obtain another proof of the fact using neighborhood persistency.

\section{Conclusion}\label{sec:conclusion}
We introduced neighborhood persistency of the linear optimization (LO) relaxation of integer linear optimization (ILO), 
which is a property stronger than persistency, 
and show that ILO on unit-two-variable-per-inequality (UTVPI) systems is a maximal subclass of ILO with its LO relaxation having (neighborhood) persistency.
Our persistency result generalizes known results on special cases of ILO on UTVPI systems~\cite{NeT75,HMN93,FJW21}.
Using neighborhood persistency, we obtain fixed-parameter algorithms (where the parameter is the solution size) and another proof of the two-approximability for special cases of ILO on UTVPI systems.
An interesting future direction will be to find a (maximal) subclass of ILO with its LO relaxation having (neighborhood) persistency that is incomparable to ILO on UTVPI systems.
Future works will also include generalizations of our result to nonlinear objective functions.



\providecommand{\noopsort}[1]{}

\appendix

\section{A proof of a variant of Theorem~\ref{thm:main}}
\label{appendix:L-extendability}
Here, we reveal a connection between ILO on UTVPI systems and discrete convex analysis on graph structures initiated by Hirai~\cite{Hir18}, and show the following theorem, which is slightly weaker than our main theorem (\cref{thm:main}) but still a generalization of the persistency results in~\cite{NeT75,HMN93,FJW21}.
For vector $x \in \mathbb{R}^n$ define its \emph{$\frac{1}{2}$-neighborhood} $N_{\frac{1}{2}}(x) \subseteq \mathbb{Z}^n$ as the set of integer vectors whose $\ell_{\infty}$ distance from $x$ is at most one half, 
i.e., $N_{\frac{1}{2}}(x)=\{ z \in \mathbb{Z}^n \mid |z_j -x_j| \le \frac{1}{2} \ (j=1,\dots, n) \}$.
Define $\frac{1}{2}\mathbb{Z} = \{ \frac{k}{2} \mid k \in \mathbb{Z} \}$.
Any vector in $(\frac{1}{2}\mathbb{Z})^n$ is called \emph{half-integral}.

\begin{theorem}
\label{thm:main-alt}
Assume that integer linear optimization~\eqref{eq:ILO} is feasible.
For every half-integral optimal solution $x^*$ of the linear optimization relaxation \eqref{eq:LO} 
there exists an optimal integer solution of \eqref{eq:ILO} in $N_{\frac{1}{2}}(x^*)$.
\end{theorem}

Note that \cref{thm:main-alt} is weaker than \cref{thm:main} since half integrality of the optimal solution of the linear optimization relaxation is assumed as in~\cite{NeT75,HMN93,FJW21}.

We show \cref{thm:main-alt} using a result on persistency in discrete convex analysis on graph structures~\cite{Hir18}.
The result is on functions defined on a general graph class and their relaxations.
However, we only need a result specialized to functions on $\mathbb{Z}^n$ (regarded as a graph) and their relaxations.
The specialized result roughly says that 
if a function on $\mathbb{Z}^n$ has a nice relaxation on $(\frac{1}{2}\mathbb{Z})^n$, then for any minimizer $x^*$ of the relaxation there exists a minimizer of the function in the $\frac{1}{2}$-neighborhood of $x^*$.
Using this result we can show \cref{thm:main-alt} as follows.
First we observe that solving the ILO problem~\eqref{eq:ILO} is equivalent to minimizing a function $h$ on $\mathbb{Z}^n$ that is the sum of the objective function and the indicator function of each constraint in the ILO problem~\eqref{eq:ILO}.
The function $h$ is shown to have a nice relaxation $g$, which is obtained by naturally extending $h$ to half-integral vectors.
Since any half-integral optimal solution $x^*$ to the LO relaxation~\eqref{eq:LO} is a minimizer of $g$, 
we obtain \cref{thm:main-alt} by the specialized result on persistency.
We state this formally below.

The nice relaxation mentioned above is defined by L-convex functions on $(\frac{1}{2}\mathbb{Z})^n$~\cite{Hir18}.
To define L-convex functions on $(\frac{1}{2}\mathbb{Z})^n$, 
we regard it as a directed graph.
Namely, we regard $\frac{1}{2}\mathbb{Z}$ as a directed graph, where the vertex set is $\frac{1}{2}\mathbb{Z}$ and 
the arc set is $\{ (x,y) \in \frac{1}{2}\mathbb{Z} \times \frac{1}{2}\mathbb{Z} \mid |x-y|=\frac{1}{2}, x \in \mathbb{Z} \}$.
For example, we have arcs $(1,1.5)$ and $(1,0.5)$.
Then $(\frac{1}{2}\mathbb{Z})^n$ denotes the directed graph defined as the Cartesian product of $n$ copies of $\frac{1}{2}\mathbb{Z}$ in what follows.

To define L-convex functions on $(\frac{1}{2}\mathbb{Z})^n$, we also need flooring and ceiling functions defined along the edge directions in directed graph $(\frac{1}{2}\mathbb{Z})^n$. 
For $x \in \frac{1}{4}\mathbb{Z}$, we define $\lfloor x \rfloor \in \frac{1}{2}\mathbb{Z}$ and $\lceil x \rceil \in \frac{1}{2}\mathbb{Z}$ as 
\begin{align*}
\lfloor x \rfloor = 
\begin{cases}
x & \mbox{if $x \in \frac{1}{2}\mathbb{Z}$}\\
x+\frac{1}{4} & \mbox{if $x-\frac{1}{4} \in \mathbb{Z}$}\\
x-\frac{1}{4} & \mbox{if $x+\frac{1}{4} \in \mathbb{Z}$}, 
\end{cases}
\end{align*}
and 
\begin{align*}
\lceil x \rceil = \begin{cases}
x & \mbox{if $x \in \frac{1}{2}\mathbb{Z}$}\\
x-\frac{1}{4} & \mbox{if $x-\frac{1}{4} \in \mathbb{Z}$}\\
x+\frac{1}{4} & \mbox{if $x+\frac{1}{4} \in \mathbb{Z}$}.
\end{cases}
\end{align*}
Hence, $\lfloor \cdot \rfloor$ (resp., $\lceil \cdot \rceil$) rounds a non half-integral value away from (resp., towards) integers.
For vector $x \in (\frac{1}{4}\mathbb{Z})^n$, $\lfloor x \rfloor$ (resp., $\lceil x \rceil$) is defined as the vector obtained by applying $\lfloor \cdot \rfloor$ (resp., $\lceil \cdot \rceil$) componentwise.
Note that $\lfloor \cdot \rfloor$ and $\lceil \cdot \rceil$ are different from those used in the body of this paper.

From Theorem 4.5 in \cite{Hir18}, we can define L-convex functions on $(\frac{1}{2}\mathbb{Z})^n$ in several equivalent ways, 
and we choose the following definition.
Let $\infty$ denote the infinity element treated as $\alpha + \infty = \infty$ and $\alpha < \infty$ for $\alpha \in \mathbb{R}$, and $\infty + \infty = \infty$.
Let $\overline{\mathbb{R}} := \mathbb{R} \cup \{\infty\}$.

\begin{definition}
\label{def:L-convexity-on-half-integers}
A function $g:(\frac{1}{2}\mathbb{Z})^n \rightarrow \overline{\mathbb{R}}$ is called \emph{L-convex} if 
\begin{align}\label{eq:L-convex-definition}
g(x) + g(y) \ge g\left(\left\lfloor \frac{x+y}{2} \right\rfloor \right) + g\left(\left\lceil \frac{x+y}{2} \right\rceil \right)
\end{align}
for each $x,y \in (\frac{1}{2}\mathbb{Z})^n$.
\end{definition}

Using L-convex functions, we define functions that have nice relaxations.

\begin{definition}
\label{def:L-extendability-on-integers}
A function $h:\mathbb{Z}^n \rightarrow \overline{\mathbb{R}}$ is called \emph{L-extendable} if there exists an L-convex function $g:(\frac{1}{2}\mathbb{Z})^n \rightarrow \overline{\mathbb{R}}$ such that the restriction of $g$ to $\mathbb{Z}^n$ coincides with $h$.
Then $g$ is called an L-convex relaxation of $h$.
\end{definition}

Now, we formally state the persistency result in \cite{Hir18} specialized to functions on $\mathbb{Z}^n$.
A function $g:(\frac{1}{2}\mathbb{Z})^n \rightarrow \overline{\mathbb{R}}$ has \emph{discrete image} if there is $\varepsilon > 0$ such that $[g(x)-\varepsilon, g(x)+\varepsilon] \cap g((\frac{1}{2}\mathbb{Z})^n) = \{g(x)\}$ for every $x$ with $g(x) < \infty$.
\begin{theorem}[Theorem 4.4 in \cite{Hir18} specialized to functions on $\mathbb{Z}^n$]\label{thm:L-persistency}
Let $h:\mathbb{Z}^n \rightarrow \overline{\mathbb{R}}$ be an L-extendable function and $g:(\frac{1}{2}\mathbb{Z})^n \rightarrow \overline{\mathbb{R}}$ an L-convex relaxation of $h$.
Suppose that $g$ has discrete image.
For any minimizer $x^*$ of $g$ (over $(\frac{1}{2}\mathbb{Z})^n$) there exists a minimizer of $h$ (over $\mathbb{Z}^n$) in $N_{\frac{1}{2}}(x^*)$.
\end{theorem}

Our aim is to show \cref{thm:main-alt} using \cref{thm:L-persistency}.
For this, we show that the objective function and the indicator function of each constraint in the ILO problem~\eqref{eq:ILO} are L-extendable in what follows.

We first show that linear functions on $\mathbb{Z}^n$ are L-extendable, 
by showing that the extension of every linear function to half-integral vectors is an L-convex function on $(\frac{1}{2}\mathbb{Z})^n$.

\begin{lemma}\label{lem:linear-L-convex}
For $w \in \mathbb{Z}^n$ 
define $g^w: (\frac{1}{2}\mathbb{Z})^n \rightarrow \overline{\mathbb{R}}$ as $g^w(x) = w^Tx$.
Then $g^w$ is an L-convex function on $(\frac{1}{2}\mathbb{Z})^n$.
\end{lemma}

\begin{proof}
From \cref{def:L-convexity-on-half-integers}, 
it suffices to show that 
\begin{align}\label{eq:linear-L-convex}
g^w(x) + g^w(y) \ge g^w\left(\left\lfloor \frac{x+y}{2} \right\rfloor \right) + g^w\left(\left\lceil \frac{x+y}{2} \right\rceil \right)
\end{align}
for each $x,y \in (\frac{1}{2}\mathbb{Z})^n$.

To show this, 
we make the following observation, which is easy to prove but useful.
For $a,b \in \frac{1}{2}\mathbb{Z}$, we have 
\begin{align*}
\left\lfloor \frac{a+b}{2} \right\rfloor + \left\lceil \frac{a+b}{2} \right\rceil
= a + b.
\end{align*}
This can be immediately extended to the sum of two vectors $x,y \in (\frac{1}{2}\mathbb{Z})^n$: 
\begin{align}\label{obs:round-up-down-linear-vector}
\left\lfloor \frac{x+y}{2} \right\rfloor + \left\lceil \frac{x+y}{2} \right\rceil
= x + y.
\end{align}

From \eqref{obs:round-up-down-linear-vector}, we have 
\begin{align}
g^w(x) + g^w(y) & =w^Tx + w^Ty \\
&= w^T(x+y)\\
&= w^T(\left\lfloor \frac{x+y}{2} \right\rfloor + \left\lceil \frac{x+y}{2} \right\rceil)\\
&= w^T\left\lfloor \frac{x+y}{2} \right\rfloor + w^T\left\lceil \frac{x+y}{2} \right\rceil\\
&=g^w\left(\left\lfloor \frac{x+y}{2} \right\rfloor \right) + g^w\left(\left\lceil \frac{x+y}{2} \right\rceil \right).
\end{align}
Hence, we obtain \eqref{eq:linear-L-convex}.
This completes the proof.
\end{proof}

From \cref{lem:linear-L-convex}, linear function $w^Tx$ on $\mathbb{Z}^n$ is L-extendable, where its L-convex relaxation can be chosen as $w^Tx$ on $(\frac{1}{2}\mathbb{Z})^n$.
Thus, we obtain the following.

\begin{corollary}\label{cor:linear-L-extendable}
For $w \in \mathbb{Z}^n$ 
define $h^w: \mathbb{Z}^n \rightarrow \overline{\mathbb{R}}$ as $h^w(x) = w^Tx$.
Then, $h^w$ is L-extendable and $g^w$ defined in \cref{lem:linear-L-convex} is its L-convex relaxation.
\end{corollary}

Now, we show that the indicator function of each constraint in the ILO problem~\eqref{eq:ILO} is L-extendable.
For this, we show that the extension of each indicator function to half-integral vectors is an L-convex function on $(\frac{1}{2}\mathbb{Z})^n$ in the following lemma.

\begin{lemma}\label{lem:UTVPI-L-convex}
For $\sigma,\tau \in \{-,+\}$, $\beta \in \mathbb{Z}$, and $p,q\in \{1,\dots,n\}$ with $p \neq q$, 
let $g^\beta_{\sigma p \tau q}: (\frac{1}{2}\mathbb{Z})^n \rightarrow \overline{\mathbb{R}}$ (resp, $g^\beta_{\sigma p}: (\frac{1}{2}\mathbb{Z})^n \rightarrow \overline{\mathbb{R}}$) be the indicator function of the constraint $\sigma x_p + \tau x_q \ge \beta$ (resp., $\sigma x_p \ge \beta$) on $(\frac{1}{2}\mathbb{Z})^n$.
Namely, 
\begin{align*}
g^\beta_{\sigma p \tau q}(x) = 
\begin{cases}
0 & \mbox{if $\sigma x_p + \tau x_q \ge \beta$}\\
\infty & \mbox{otherwise},
\end{cases}
\end{align*}
and 
\begin{align*}
g^\beta_{\sigma p}(x) = 
\begin{cases}
0 & \mbox{if $\sigma x_p \ge \beta$}\\
\infty & \mbox{otherwise},
\end{cases}
\end{align*}
Then $g^\beta_{\sigma p \tau q}$ and $g^\beta_{\sigma p}$ are L-convex functions on $(\frac{1}{2}\mathbb{Z})^n$.
\end{lemma}

\begin{proof}
Fix $\sigma,\tau \in \{-,+\}$, $\beta \in \mathbb{Z}$, and $p,q\in \{1,\dots,n\}$ with $p \neq q$.
We only show L-convexity of $g^\beta_{\sigma p \tau q}$; L-convexity of $g^\beta_{\sigma p}$ can be proven similarly.
We denote $g^\beta_{\sigma p \tau q}$ by $g$ for simplicity.
From \cref{def:L-convexity-on-half-integers}, 
it suffices to show that 
\begin{align}
g(x) + g(y) \ge g\left(\left\lfloor \frac{x+y}{2} \right\rfloor \right) + g\left(\left\lceil \frac{x+y}{2} \right\rceil \right)
\end{align}
for each $x,y \in (\frac{1}{2}\mathbb{Z})^n$.
This is equivalent to that 
$g(x)=g(y)=0 \Rightarrow g\left(\left\lfloor \frac{x+y}{2} \right\rfloor \right) = g\left(\left\lceil \frac{x+y}{2} \right\rceil \right) = 0$, which we will show in what follows.
Assume that $g(x)=g(y)=0$, i.e., $\sigma x_p + \tau x_q \ge \beta$ and $\sigma y_p + \tau y_q \ge \beta$.
We have to show that 
\begin{align}
&\sigma \left\lfloor \frac{x_p+y_p}{2} \right\rfloor + \tau \left\lfloor \frac{x_q+y_q}{2} \right\rfloor \ge \beta\mbox{, and}\label{eq:floor}\\
&\sigma \left\lceil \frac{x_p+y_p}{2} \right\rceil + \tau \left\lceil \frac{x_q+y_q}{2} \right\rceil \ge \beta.\label{eq:ceil}
\end{align}
From the assumption, we have $\sigma \frac{x_p+y_p}{2} + \tau \frac{x_q+y_q}{2} \ge \beta$.
If $\frac{x_p+y_p}{2}, \frac{x_q+y_q}{2} \in \frac{1}{2}\mathbb{Z}$, then $\left\lfloor \frac{x_p+y_p}{2} \right\rfloor = \left\lceil \frac{x_p+y_p}{2} \right\rceil = \frac{x_p+y_p}{2}$
and $\left\lfloor \frac{x_q+y_q}{2} \right\rfloor = \left\lceil \frac{x_q+y_q}{2} \right\rceil = \frac{x_q+y_q}{2}$.
Hence, Inequalities \eqref{eq:floor} and \eqref{eq:ceil} hold.
Assume without loss of generality that $\frac{x_p+y_p}{2} \not\in \frac{1}{2}\mathbb{Z}$.
We divide into three cases according to the value $\sigma \frac{x_p+y_p}{2} + \tau \frac{x_q+y_q}{2}$.

\paragraph*{Case 1: $\sigma \frac{x_p+y_p}{2} + \tau \frac{x_q+y_q}{2} \ge \beta + \frac{1}{2}$.}
Taking $\lfloor \cdot \rfloor$ (resp., $\lceil \cdot \rceil$) loses at most $\frac{1}{4}$.
Hence, $\sigma \left\lfloor \frac{x_p+y_p}{2} \right\rfloor + \tau \left\lfloor \frac{x_q+y_q}{2} \right\rfloor \ge \sigma \frac{x_p+y_p}{2} - \frac{1}{4} + \tau \frac{x_q+y_q}{2} - \frac{1}{4} \ge \beta$, obtaining \eqref{eq:floor}.
Similarly, we have \eqref{eq:ceil}.

\paragraph*{Case 2: $\sigma \frac{x_p+y_p}{2} + \tau \frac{x_q+y_q}{2} = \beta + \frac{1}{4}$.}
From $\frac{x_p+y_p}{2} \not\in \frac{1}{2}\mathbb{Z}$, we have 
$\tau \frac{x_q+y_q}{2} = \beta + \frac{1}{4} - \sigma \frac{x_p+y_p}{2} \in \frac{1}{2}\mathbb{Z}$.
Hence, $\tau \left\lfloor \frac{x_q+y_q}{2} \right\rfloor = \tau \left\lceil \frac{x_q+y_q}{2} \right\rceil = \tau \frac{x_q+y_q}{2}$.
Therefore, from $\sigma \left\lfloor \frac{x_p+y_p}{2} \right\rfloor, \sigma \left\lceil  \frac{x_p+y_p}{2} \right\rceil \ge \sigma \frac{x_p+y_p}{2} - \frac{1}{4}$, 
we have \eqref{eq:floor} and \eqref{eq:ceil}.

\paragraph*{Case 3: $\sigma \frac{x_p+y_p}{2} + \tau \frac{x_q+y_q}{2} = \beta$.}
In this case, we have $\sigma \left\lfloor \frac{x_p+y_p}{2} \right\rfloor = \sigma \frac{x_q+y_q}{2} + \frac{1}{4}$ if and only if $\tau \left\lfloor \frac{x_q+y_q}{2} \right\rfloor = \tau \frac{x_q+y_q}{2} - \frac{1}{4}$.
Hence, we have \eqref{eq:floor}.
From $\sigma \left\lfloor \frac{x_p+y_p}{2} \right\rfloor + \sigma \left\lceil  \frac{x_p+y_p}{2} \right\rceil = \sigma (x_p+y_p)$ and $\tau \left\lfloor \frac{x_q+y_q}{2} \right\rfloor + \tau \left\lceil  \frac{x_q+y_q}{2} \right\rceil = \tau (x_q+y_q)$, we also have \eqref{eq:ceil}.

Hence, $g$ is L-convex on $(\frac{1}{2}\mathbb{Z})^n$.
\end{proof}

From \cref{lem:UTVPI-L-convex}, 
we immediately have the following.

\begin{corollary}\label{cor:UTVPI-L-extendable}
For $\sigma,\tau \in \{-,+\}$, $\beta \in \mathbb{Z}$, and $p,q\in \{1,\dots,n\}$ with $p \neq q$, 
let $h^\beta_{\sigma p \tau q}: \mathbb{Z}^n \rightarrow \overline{\mathbb{R}}$ (resp, $h^\beta_{\sigma p}: \mathbb{Z}^n \rightarrow \overline{\mathbb{R}}$) be the indicator function of the constraint $\sigma x_p + \tau x_q \ge \beta$ (resp., $\sigma x_p \ge \beta$) on $\mathbb{Z}^n$.
Then $h^\beta_{\sigma p \tau q}$ (resp, $h^\beta_{\sigma p}$) is L-extendable and $g^\beta_{\sigma p \tau q}$ (resp., $g^\beta_{\sigma p}$) defined in \cref{lem:UTVPI-L-convex} is its L-convex relaxation.
\end{corollary}

\begin{remark}
The function $h^\beta_{\sigma p \tau q}$ defined in \cref{cor:UTVPI-L-extendable} is \emph{not} in general a ``coarsening'' of an L-convex function on $(\frac{1}{2}\mathbb{Z})^n$, 
i.e., an L-convex function on $\mathbb{Z}^n$ where $\mathbb{Z}^n$ is regarded as a directed graph with the vertex set $\mathbb{Z}^n$ and the arc set $\{ (x,y) \in \mathbb{Z} \times \mathbb{Z} \mid |x-y|=1, x \text{ is even} \}$.
We note that such an L-convex function is called a UJ-convex function in~\cite{Fuj14}.
\end{remark}

Now, we are ready to show \cref{thm:main-alt} using \cref{thm:L-persistency}.

\begin{proof}[Proof of \cref{thm:main-alt}]
Assume that the ILO problem~\eqref{eq:ILO} is feasible.
For the ILO problem~\eqref{eq:ILO}, 
define $h^w: \mathbb{Z}^n \rightarrow \overline{\mathbb{R}}$ as $h(x)=w^Tx$ and $h_i: \mathbb{Z}^n \rightarrow \overline{\mathbb{R}}$ as the indicator function of $A_ix \ge b_i$ for each $i = 1,\dots,m$.
Define $h = h^w + \sum_{i=1}^{m}h_i$.
Then $h$ is L-extendable and one can choose as its L-convex relaxation 
the L-convex function $g$ obtained by naturally extending $h$ to half-integral vectors by \cref{cor:linear-L-extendable,cor:UTVPI-L-extendable}, 
since the non-negative sum of L-convex functions is again L-convex by definition.
Clearly, $g$ has discrete image.
Let $x^*$ be a half-integral optimal solution of the LO relaxation~\eqref{eq:LO}.
Then $x^*$ is a minimizer of $g$.
Hence, from \cref{thm:L-persistency}, there exists a minimizer of $h$ in $N_{\frac{1}{2}}(x^*)$.
Since any minimizer of $h$ is an optimal integer solution of the ILO problem~\eqref{eq:ILO}, 
we obtain \cref{thm:main-alt}.
\end{proof}

\end{document}